\documentclass[12pt, reqno]{amsart}
\usepackage{graphicx}
\usepackage{hyperref}
\usepackage[caption = false]{subfig}
\usepackage{amsthm}
\usepackage{amssymb}
\usepackage{mathrsfs}
\usepackage{mathtools}
\usepackage{enumerate}
\usepackage{amssymb}
\usepackage{wrapfig}
\usepackage{float}
\allowdisplaybreaks

\usepackage[twoside,paperwidth=190mm, paperheight=297mm, top=35mm, bottom=35mm, left=30mm, right=26mm, marginparsep=3mm, marginparwidth=40mm]{geometry}

\numberwithin{equation}{section}

\theoremstyle{plain}

\newtheorem{theorem}{Theorem}[section]
\newtheorem{corollary}[theorem]{Corollary}
\newtheorem{example}[theorem]{Example}
\newtheorem{lemma}{Lemma}[section]

\theoremstyle{definition}
\newtheorem{definition}[theorem]{Definition}
\theoremstyle{remark}
\newtheorem{remark}{Remark}[section]
\newtheorem{problem}[theorem]{Problem}
\newtheorem{question}[theorem]{Question}

\makeatother

\begin{document}
	
	\title{A generalized Bohr-Rogosinski phenomenon}
	\thanks{K. Gangania thanks to University Grant Commission, New-Delhi, India for providing Research Fellowship under UGC-Ref. No.:1051/(CSIR-UGC NET JUNE 2017).}

	\author[Kamaljeet]{Kamaljeet Gangania}
	\address{Department of Applied Mathematics, Delhi Technological University,
		Delhi--110042, India}
	\email{gangania.m1991@gmail.com}
	
	\author[S. Sivaprasad Kumar]{S. Sivaprasad Kumar}
	\address{Department of Applied Mathematics, Delhi Technological University,
		Delhi--110042, India}
	\email{spkumar@dce.ac.in}

	\maketitle	
	
	\begin{abstract} 
		In this paper, we generalize the Bohr-Rogosinski sum for the Ma-Minda classes of starlike and convex functions. Also the phenomenon is studied for the classes of starlike functions with respect to symmetric points and conjugate points along with their convex cases. Further, the connection between the derived results and the known ones are established with the suitable examples.
	\end{abstract}
	\vspace{0.5cm}
	
	\maketitle
	
	\section{Introduction}
	\label{intro}
	Let $\mathcal{A}$ denote the class of analytic functions of the form $f(z)=z+\sum_{k=2}^{\infty}a_kz^k$ in the open unit disk $\mathbb{D}:=\{z: |z|<1\}$. We denote by $\mathcal{S}$, the subclass of $\mathcal{A}$ of univalent functions. Using subordination~\cite{subbook}, Ma and Minda \cite{minda94} (also see \cite{ganga-iranian}) introduced the unified class of univalent starlike and convex functions defined as follows:
	\begin{equation*}
	\mathcal{S}^*(\psi):= \biggl\{f\in \mathcal{A} : \frac{zf'(z)}{f(z)} \prec \psi(z) \biggl\}
	\end{equation*}
	and
	\begin{equation*}\label{mindaclass}
	\mathcal{C}(\psi):= \biggl\{f\in \mathcal{A} : 1+\frac{zf''(z)}{f'(z)} \prec \psi(z) \biggl\},
	\end{equation*}
	where  $\psi$ is univalent, analytic and sarlike with respect to $1$ with $\Re{\psi(z)}>0$, $\psi'(0)>0$, $\psi(0)=1$ and $\psi(\mathbb{D})$ is symmetric about real axis. Note that $\psi \in \mathcal{P}$, the class of Carath\'{e}odory functions. Also when $\psi(z)=(1+z)/(1-z)$, $\mathcal{S}^*(\psi)$  and $\mathcal{C}(\psi)$ reduce to the standard classes $\mathcal{S}^*$ and $\mathcal{C}$ of univalent starlike and convex functions respectively. 
	
	In GFT, radius problems have a long and glorious history that continues to this day, see the recent articles \cite{ganga_cmft2021,janow,Kumar-cardioid,ganga-iranian,kumarG-2020}. In 1914, Harald Bohr \cite{bohr1914} proved a significant radius problem for the absolute power series:
	\begin{theorem}[Bohr's Theorem, \cite{bohr1914}]\label{BohrTheorem}
		Let $g(z)=\sum_{k=0}^{\infty}a_kz^k$ be an analytic function in $\mathbb{D}$ and $|g(z)|<1$ for all $z\in \mathbb{D}$, then
		\begin{equation*}
		\sum_{k=0}^{\infty}|a_k||z|^k\leq1, \quad \text{for} \quad |z|\leq\frac{1}{3}.
		\end{equation*}
	\end{theorem}

	In analogy with Bohr's Theorem, there is also the notion of Rogosinski radius, however a little is known about Rogosinski radius (also see \cite{Landau-1986,Rogosinski-1923,Schur-1925}) as compared to Bohr radius, which is defined as follows:
	\begin{theorem}[Rogosinski Theorem]
		If $g(z)=\sum_{k=0}^{\infty}b_k z^k$ with $|g(z)|<1$, then for every $N\geq1$, we have
		\begin{equation*}
		\left| \sum_{k=0}^{N-1}b_k z^k \right| \leq1, \quad\text{for}\quad |z|\leq \frac{1}{2}.
		\end{equation*}
		The radius $1/2$ is called the Rogosinski radius.
	\end{theorem}
	
	Recently, Kumar and Sahoo~\cite{Kumar-shaoo-genBohr} obtained the generalized classical Bohr's Theorem for functions satisfying $\Re{(f)(z)}\leq1$. Also see, Kayumov et al.\cite{Kayumov-gen-cesaro}.  
	\begin{theorem}\cite[Theorem~2.2]{Kumar-shaoo-genBohr}\label{Shaoo-Theorem}
		Let $\{\nu_{k}(r)\}_{k=0}^{\infty}$ be a sequence of non-negative continuous functions in $[0,1)$ such that the series 
		\begin{equation*}
		\nu_0(r) + \sum_{k=1}^{\infty}\nu_{k}(r)
		\end{equation*}
		converges locally uniformly with respect to $r\in [0,1)$. Let $f(z)=\sum_{n=0}^{\infty} a_nz^n$ with $\Re{f(z)} \leq 1$ and $p\in (0,1]$. If 
		\begin{equation*}
		\nu_0(r) > \frac{2(1+r^m)}{p(1-r^m)} \sum_{k=1}^{\infty}\nu_{k}(r).
		\end{equation*}
		Then the following sharp inequality holds:
		\begin{equation*}
		|f(z^m)|\nu_0(r) + \sum_{k=1}^{\infty}|a_k| \nu_{k}(r) \leq \nu_0(r) \quad \text{for all} \quad |z|=r\leq R_1,
		\end{equation*}
		where $R_1$ is the minimal postive root of the equation:
		\begin{equation*}
		\nu_0(r) = \frac{2(1+r^m)}{p(1-r^m)} \sum_{k=1}^{\infty}\nu_{k}(r).
		\end{equation*}
		In case when 
		\begin{equation*}
		\nu_0(r) < \frac{2(1+r^m)}{p(1-r^m)} \sum_{k=1}^{\infty}\nu_{k}(r)
		\end{equation*}
		in some interval $(R_1, R_1+\epsilon)$, then the number $R_1$ can not be improved.
	\end{theorem}
	
	If we choose $\nu_k(r)=r^k$ in Theorem~\ref{Shaoo-Theorem}, we get Theorem~\ref{BohrTheorem}. At this juncture, it is natural to pose the following problem:
	\begin{problem}\label{Gen-BohrRogProb}
		Can we establish the analogue of Theorem\ref{Shaoo-Theorem} for the Ma-Minda classes $\mathcal{S}^*(\psi)$ and $\mathcal{C}(\psi)$ ?
	\end{problem}	
	
	In context of the above problem and Muhanna~\cite{muhanna10}, we now describe the notion of generalized Bohr-Rogosinski phenomenon here below, in terms of subordination, following the recent development as seen in \cite{Kayumov-gen-cesaro, Kumar-shaoo-genBohr, Lin-Liu-genBohr}.
	\begin{definition}\label{Def-GenBohr} 
		Let $f(z)=\sum_{k=0}^{\infty}a_kz^k$ and $g(z)=\sum_{k=0}^{\infty}b_kz^k$ are analytic in $\mathbb{D}$. Let $d(f(0),\partial\Omega)$ denotes the Euclidean distance between $f(0)$ and the boundary of $\Omega=f(\mathbb{D})$. For a fixed $f$, consider a class of analytic functions 	
		$$S(f):=\{g : g\prec f\}$$ 
		or equivalently, 
		$$S(\Omega):=\{g : g(z)\in \Omega\}.$$
		Then  we say $S(f)$ satisfies the {\it Generalized Bohr-Rogosinski phenomenon}, if there exists a constant $r_0\in (0,1]$ such that 
		\begin{equation*}
		\mathcal{P}(r,g,f) + \sum_{k=1}^{\infty}|b_k| \phi_{k}(r) \leq d(f(0), \partial \Omega),
		\end{equation*}
		holds for all $|z|=r\leq r_0$, where
		\begin{enumerate}
			\item   $\mathcal{P}(r,g,f)$ represent some function of $r$ or certain proper combination of moduli of $g$, $f$ and  their derivatives. 
			\item   $\{\phi_{k}(r)\}$ be a sequence of non-negative continuous functions in $[0,1)$ such that the series of the form
			\begin{equation*}
			p_0\phi_0(r) + \sum_{k=1}^{\infty}p_{k}\phi_{k}(r)
			\end{equation*}
			converges locally uniformly with respect to $r\in [0,1)$, where $p_{k}$ depends on the function $f$ and provide bounds for $b_k$.
		\end{enumerate}
	\end{definition}
	
	For $\mathcal{P}(r,g,f)=|g(z)|$, $\phi_{k}(r)=r^k$ $(k\geq N)$ and $0$ otherwise in the Definition~\ref{Def-GenBohr} gives the quantity considered by Kayumov et al.~\cite{Kayumov-2021}, which is known as the Bohr-Rogosinski sum, given by
	\begin{equation*}
	|g(z)|+ \sum_{k=N}^{\infty}|b_k||z|^k, \quad |z|=r.
	\end{equation*}
	The link between the Bohr-Rogosinski and Bohr phenomenon can be noticed, if we replace $|g(z)|$ by $g(0)$ with $N=1$. We also refer the readers to see~\cite{Aizenberg-2012,Alkha-2020}. Now we see that the family $S(f)$ has Bohr-Rogosinski phenomenon provided there exists $r^{f}_{N} \in (0,1]$ such that the inequality: 
	$$|g(z)|+ \sum_{k=N}^{\infty}|b_k||z|^k \leq |f(0)|+ d(f(0),\partial \Omega)$$
	holds for $|z|=r \leq r^{f}_{N}$. The largest such $r^{f}_{N}$ is called the Bohr-Rogosinski radius.
	
	In case when the function $f$ is normalized, then Kumar and Gangania~\cite{kamal-mediter2021} considered the class
	\begin{definition}
		Let $f\in \mathcal{S}^*(\psi)$ or $\mathcal{C}(\psi)$ be fixed. Then the class of subordinants functions $g$ is defined as:
		\begin{equation*}\label{bohrclass}
		S_{f}(\psi):= \biggl\{g(z)=\sum_{k=1}^{\infty}b_k z^k : g \prec f \biggl \}.
		\end{equation*}
	\end{definition}
	\noindent and studied the Bohr-Rogosinski phenomenon for the class of analytic subordinants $S_{f}(\psi)$:
	
	The class $S_{f}(\psi)$ has a Bohr-Rogosinski phenomenon, if there exists an $0<r_0 \leq1$ such that
	\begin{equation*}
	|g(z^m)|+\sum_{k=N}^{\infty}|b_k||z|^k \leq d(f(0), \partial \Omega)
	\end{equation*}
	for $|z|=r\leq r_0$, where $m, N\in \mathbb{N}$, $\Omega=f(\mathbb{D})$ and  $d(f(0),\partial\Omega)$ denotes the Euclidean distance between $f(0)$ and the boundary of $\Omega$.
	
	\begin{theorem}\cite[Theorem~2.3, Page no.~7]{kamal-mediter2021}
		Let $f_0(z)$ be given by the equation~\eqref{int-rep} and $f(z)=z+\sum_{n=2}^{\infty}a_n z^n \in \mathcal{S}^*(\psi)$. Assume  $f_0(z)=z+\sum_{n=2}^{\infty}t_n z^n$ and $\hat{f}_0(r)=r+\sum_{n=2}^{\infty}|t_n|r^n$. If $g\in S_{f}(\psi)$. Then 
		\begin{equation*}
		|g(z^m)| + \sum_{k=N}^{\infty}|b_k||z|^k \leq d(0, \partial{\Omega})
		\end{equation*}
		holds for $|z|=r_b \leq \min \{ \frac{1}{3}, r_0 \}$, where $m, N\in \mathbb{N}$, $\Omega=f(\mathbb{D})$ and $r_0$ is the unique positive root of the equation:
		\begin{equation*}
		\hat{f}_0 (r^m) + \hat{f}_0 (r) -p_{\hat{f}_0}(r)=-f_0(-1),
		\end{equation*}
		where 
		\begin{equation*}
		p_{\hat{f}_0}(r)=
		\left\{
		\begin{array}
		{lr}
		0, & N=1; \\
		r,   & N=2;\\
		r+\sum_{n=2}^{N-1}|t_n|r^n, & N\geq3.
		\end{array}
		\right.
		\end{equation*}
		The result is sharp when $r_b=r_0$ and $t_n>0$.
	\end{theorem}
	For the class $\mathcal{C}(\psi)$, see \cite[Theorem~2.13, page no.~12]{kamal-mediter2021}. In context of the above, also see \cite{Kayumov-2021,Aizenberg-2012,Alkha-2020,bhowmik2018}.
	
	The Bohr Operator
	\begin{equation*}
	M_r(f)= \sum_{n=0}^{\infty}|a_n||z^n|=  \sum_{n=0}^{\infty}|a_n| r^n
	\end{equation*}
	for the analytic functions $f(z)=\sum_{n=0}^{\infty}a_n z^n$ was used by
	Paulsen and Singh~\cite{paulsen-Singh} to provide a simple proof of the Bohr's Theorem~\ref{BohrTheorem} and extended it to the Banach algebras (for the basic important discussion, see \cite{Muhanna-2021,paulsen-Singh}). Later, Muhanna et al.~\cite{Muhanna-2021} used this operator to obtain some new Bohr type of inequalities for the $k$-th section $\sum_{n=0}^{k}a_n z^n$ and also obtain simple proofs of the known results. 
	
	Gangania and Kumar~\cite{kamal-mediter2021} considered the basic operator for $f$, given by:
	\begin{equation*}
	M^{N}_r(f)= \sum_{n=N}^{\infty}|a_n||z^n|=  \sum_{n=N}^{\infty}|a_n| r^n,
	\end{equation*}
	and thus the following observations hold for $|z|=r$ for each $z\in \mathbb{D}$
	\begin{enumerate}[$(i)$]
		\item $M^{N}_r(f)\geq0$, and $M^{N}_r(f)=0$ if and only if $f\equiv0$
		\item $M^{N}_r(f+g)\leq M^{N}_r(f)+M^{N}_r(g)$
		\item $M^{N}_r(\alpha f) = |\alpha| M^{N}_r(f)$ for $\alpha\in\mathbb{C}$
		\item $M^{N}_r(f.g) \leq M^{N}_r(f). M^{N}_r(g)$
		\item $M^{N}_r(1)=1.$
	\end{enumerate}
	For the above properties in terms of the sequence $\{\nu_n(r) \}_{n=0}^{\infty}$, see \cite[Lemma~1]{PonnuVijarWrith-2021}. Using this operator, a simple proof of \cite[Lemma~1]{bhowmik2018} was achieved by Gangania and Kumar~\cite{kamal-mediter2021} to settle the Bohr-Rogosinski Phenomenon for the classes $\mathcal{S}^*(\psi)$ and $\mathcal{C}(\psi)$, which in terms of interim $k$-th section $f_{k}(z):=\sum_{n=N}^{k}a_n z^n$ is as follows:
	
	\begin{lemma}\cite{kamal-mediter2021}\label{series-lem}
		let $f(z)=\sum_{n=0}^{\infty}a_n z^n$ and $g(z)=\sum_{k=0}^{\infty}b_k z^k$ be analytic in $\mathbb{D}$ and $g\prec f$, then 
		\begin{equation}\label{N-inequality}
		M^{N}_r(g_k) \leq M^{N}_r(f_k)
		\end{equation}
		for all $|z|=r\leq 1/3$ and $k, N\in \mathbb{N}$.
	\end{lemma}
	when $N=1$ and $k\rightarrow \infty$, the lemma was obtained by Bhowmik and Das~\cite{bhowmik2018}. Various interesting applications of this lemma can be seen  in \cite{kamal-mediter2021,bhowmik2018,hamada-2021,ganga_cmft2021,ganga-iranian}.
	
	While we establish the Bohr-type inequalities for the general class $ \mathcal{S}^*(\psi)$ or $\mathcal{C}(\psi)$, the main difficulty that we come across is the unavailability of the sharp coefficients bounds. Here, we require use of the Lemma~\ref{series-lem} or its proper modifications. Interestingly, the Lemma~\ref{series-lem} also implies that if $f\prec g$, then within the disk $|z|\leq 1/3$, we have $|a_n| \leq  |b_n|$ for all $n\in \mathbb{N}\cup \{0\}$, where $b_n$ are the coefficients of the function $g$ in Lemma~\ref{series-lem}. Further, this readily gives the following: 
	\begin{lemma}\label{series-lem-gen}
		Let $f(z)=\sum_{n=0}^{\infty}a_n z^n$ and $g(z)=\sum_{n=0}^{\infty}b_n z^n$ be analytic in $\mathbb{D}$. Let $\{\nu_{k}(r)\}_{k=0}^{\infty}$ be a sequence of non-negative functions, continuous in $[0,1)$ such that the series 
		\begin{equation*}
		\sum_{n=0}^{\infty}|b_n|\nu_{k}(r)
		\end{equation*}
		converges locally uniformly with respect to $r\in [0,1)$. If $g\prec f$, then 
		\begin{equation*}
		\sum_{n=0}^{\infty} |a_n| \nu_n(r) \leq \sum_{n=0}^{\infty} |b_n| \nu_n(r)
		\end{equation*}
		for all $|z|=r\leq \frac{1}{3}$.
	\end{lemma}
	
	For a different version of the Lemma~\ref{series-lem-gen} with the conditions $\nu_{m+n}(r) \leq \nu_m(r) \nu_n(r)$ and $\nu_0(r)=1$, see \cite[Theorem~3]{PonnuVijarWrith-2021}.
	
	A function $f\in \mathcal{A}$ is in the class $\mathcal{K}$ of close-to-convex if $\Re(zf'(z)/g(z))>0$, where $g\in \mathcal{S}^*$. In 2004, Ravichandran studied the amalgamated treatment for the classes of starlike and convex functions with respect to the symmetric points for the growth and distortion theorems:
	\begin{definition}\cite{RaviConjugate-2004}\label{symmetricpoint}
		A function $f\in \mathcal{A}$ is in the class $\mathcal{S}_{s}^{*}(\psi)$ and $\mathcal{C}_{s}(\psi)$  if it satisfy 
		$zf'(z)/h(z) \prec \psi(z)$ and $(zf'(z))'/h'(z) \prec \psi(z)$ 
		respectively, where $2h(z)=f(z)-f(-z)$.
	\end{definition}
	
	Now if we take $h(z)=f(z)+\overline{f(\bar{z})}$ in Definition~\ref{symmetricpoint}, we obtain the classes $\mathcal{S}_{c}^{*}(\psi)$ and $\mathcal{C}_{c}^{*}(\psi)$ of starlike and convex functions with respect to conjugate points, respectively. For the choice $h(z)=f(z)-\overline{f(-\bar{z})}$, we have the classes $\mathcal{S}_{sc}^{*}(\psi)$ and $\mathcal{C}_{sc}^{*}(\psi)$ of starlike and convex functions with respect to the symmetric conjugate points, respectively. See \cite{RaviConjugate-2004}.
	
	Motivated, by the class $\mathcal{S}_{s}^{*}((1+z)/(1-z))$~\cite{sakaguchi-1959}, Gao and Zhou~\cite{GaoZhou-2005} studied the class $\mathcal{K}_{s}$ of close-to-convex functions $f$, which is characterized as:
	\begin{equation*}
	\Re \left( \frac{z^2f'(z)}{g(z)g(-z)} \right) >0,
	\end{equation*} 
	where $g$ is some starlike function of order $1/2$. In view of the Definition~\ref{symmetricpoint}, the generalized class $\mathcal{K}_{s}(\psi)$ was studied by Cho et al.~\cite{ChoKwonRavi-2011} and Wang et al.~\cite{WangGaoYuan-2006}.
	
	In this paper, we positively answer the Problem~\ref{Gen-BohrRogProb} in Section~\ref{sec-1Gen} for the class $\mathcal{S}^{*}(\psi)$. In Section~\ref{sec-2Gen}, we study the Bohr-Rogosinski phenomenon for the classes $\mathcal{S}_{s}^{*}(\psi)$, $\mathcal{S}_{c}^{*}(\psi)$, $\mathcal{S}_{sc}^{*}(\psi)$ and corresponding convex classes. For convenience, we denote $\hat{f}(z)=\sum_{n=0}^{\infty}|a_n|z^n$, whenever $f(z)=\sum_{n=0}^{\infty} a_n z^n$.

	\section{Generalized Bohr's sum for Ma-Minda starlike functions}\label{sec-1Gen}
	We here solve the Problem~\ref{Gen-BohrRogProb}. But as we do not have sharp coefficient's bound for each $a_n$ for the given class in general. Thus to solve it, we need the following:
	\begin{lemma}\label{normalcompact}
		The families $\mathcal{S}^*(\psi)$ and $\mathcal{C}(\psi)$ are normal and compact.	
	\end{lemma}
	\begin{proof}
		From Montel's Theorem~\cite{goddman-1}, we see that the class $\mathcal{S}^*(\psi)$ is a normal family. Now let us prove that $\mathcal{S}^*(\psi)$ is compact.
		Let $\{f_n\}_{n\in \mathbb{N}}$ be a sequence of functions from $\mathcal{S}^*(\psi)$. Suppose that $\{f_n\}$ be convergent. Then it is well-known that 
		\begin{equation*}
		\lim_{n\rightarrow \infty} f_n := f \in \mathcal{S}.
		\end{equation*}
		We show that $f\in \mathcal{S}^*(\psi)$. If possible, suppose that there exists a nonzero point $z_0\in \mathbb{D}$ such that 
		$$g(z_0)\not \in \psi(\mathbb{D}),$$
		where $g(z)=zf'(z)/f(z)$. Note that the corresponding sequence $\{ g_n\}$ converges to $g$, where $g_n(z)=zf'_{n}(z)/f_n(z)$. Now let
		\begin{equation*}
		\epsilon = dist (g(z_0), \partial\psi(\mathbb{D})).
		\end{equation*}
		Then the open ball 
		$$B(g(z_0), \epsilon) \not\subset \psi(\mathbb{D}).$$
		Since $g_n \rightarrow g$, in particular, $g_n(z_0) \rightarrow g(z_0)$. There exists $n(\epsilon) \in \mathbb{N}$
		such that
		\begin{equation*}
		g_n(z_0) \in  B( g(z_0), \epsilon), \quad \forall n\geq n(\epsilon),
		\end{equation*}
		which implies 
		\begin{equation*}
		g_n(z_0) \not \in  \psi(\mathbb{D}), \quad \forall n\geq n(\epsilon).
		\end{equation*}
		But as $f_n \in \mathcal{S}^*(\psi)$,
		\begin{equation*}
		g_n(z_0) =\frac{z_0 f'_n(z_0)}{f_n(z_0)} \in \psi(\mathbb{D}), \quad \forall n.
		\end{equation*}
		Hence, we must have $f \in \mathcal{S}^*(\psi)$, that means the family $\mathcal{S}^*(\psi)$ is compact. With the similar arguments, it is easy to see that the family $\mathcal{C}(\psi)$ is also compact. \qed
	\end{proof}
	
	\begin{remark}{(\it Existence of sharp coefficients Bounds)}\label{variational-technique}
		Let us consider the real-valued functional $\mathcal{J}$ defined on $\mathcal{S}^*(\psi)$ as
		\begin{equation*}
		\mathcal{J}(f)= \max \{|a_n| \} \quad \text{\it for every} \quad f\in \mathcal{S}^*(\psi),
		\end{equation*}
		where $n$ is fixed and $f(z)=z+\sum_{n=2}^{\infty}a_nz^n$. From Lemma~\ref{normalcompact}, $\mathcal{S}^*(\psi)$  is normal and compact. 
		Further, since $\mathcal{S}^*(\psi) \subseteq \mathcal{S}^*$, we have $|a_n|\leq n$, that means $\mathcal{J}$ is a bounded functional. Hence, following the discussion in the Goodman's Book~\cite[page no.~44-45]{goddman-1}, we conclude that
		\begin{equation*}
		\mathcal{J}(f)= \max \{|a_n| \} \quad \text{\it exists in the family} \quad \mathcal{S}^*(\psi).
		\end{equation*} 
		Thus, let us say that
		$$\max_{f\in \mathcal{S}^*(\psi) }^{} \{|a_n| \}:= M(n)$$ 
		for each $n\in \mathbb{N}$. For instance, $M(n)=n$ for the class of univalent starlike functions. For the Janowski starlike functions, it is given by \cite[Theorem~3]{Aouf-1987}.
	\end{remark}
	
	We can now state our result in a general setting whose complement is Theorem~\ref{Gen-Mainthm2}:
	\begin{theorem}\label{Gen-Mainthm1}
		Let $\{\phi_{n}(r)  \}_{n=1}^{\infty}$ be a sequence of non-negative continuous functions in $(0,1)$ such that the series
		\begin{equation*}
		\phi_{1}(r)+\sum_{n=2}^{\infty} M(n)\phi_{n}(r)
		\end{equation*}
		converges locally uniformly with respect to each $r\in[0,1)$.
		If for $\beta\in[0,1]$
		\begin{equation}\label{Gen-condition}
		\beta f'_{0}(r^m)+ (1-\beta)f_{0}(r^m) +\sum_{n=1}^{\infty} M(n) \phi_{n}(r) < -f_{0}(-1).
		\end{equation}
		and the function $f(z)=z+\sum_{n=2}^{\infty}a_n z^n \in \mathcal{S}^*(\psi)$.
		Then the following inequality
		\begin{equation*}
		\beta |f'(z^m)| +(1-\beta)|f(z^m)| + \sum_{n=1}^{\infty}|a_n|\phi_{n}(r) \leq d(0, \partial{\Omega})
		\end{equation*}
		holds for $|z|=r\leq r_0$, where $m\in \mathbb{N}$, $\Omega=f(\mathbb{D})$ and $r_0$ is the smallest positive root of the equation:
		\begin{equation}\label{Gen-GBRequation1}
		\beta f'_{0}(r^m)+ (1-\beta)f_{0}(r^m) +\sum_{n=1}^{\infty} M(n) \phi_{n}(r) = -f_{0}(-1),
		\end{equation}
		where $M(1)=1$ and
		\begin{equation*}
		f_0(z)= z\exp\int_{0}^{z}\frac{\psi(t)-1}{t}dt.
		\end{equation*}
		In case when $f_0$ or it's rotation serves as an extremal for the coefficient's bounds $M(n)$, then the radius $r_0$ is sharp. 
	\end{theorem}
	\begin{proof}
		From Lemma~\ref{normalcompact} and Remark~\ref{variational-technique}, we see that: (a) $\mathcal{J}$ is a bounded real valued continuous functional, (b) $\mathcal{S}^*(\psi)$ is a normal family, and (c) $\mathcal{S}^*(\psi)$ is a compact family in $\mathbb{D}$. Thus, the sharp bounds for each $a_n$ exists. In view of Remark~\ref{variational-technique}, we have
		\begin{equation*}
		|a_n| \leq M(n),
		\end{equation*}
		which yields that
		\begin{equation}\label{Gen-sum1}
		\sum_{n=1}^{\infty}|a_n|\phi_{n}(r) \leq \sum_{n=1}^{\infty}M(n)\phi_{n}(r).
		\end{equation}
		The Koebe-radius for the functions in $\mathcal{S}^*(\psi)$ satisfies
		\begin{equation}\label{keobe-inequality}
		d(0, \partial f(\mathbb{D})) \geq -f_0(-1).
		\end{equation}
		Now combining it with the growth and distortion theorems~\cite{minda94}, and using the condition~\ref{Gen-condition}, the inequalities \eqref{Gen-sum1} and \eqref{keobe-inequality} gives
		\begin{align*}
		\beta |f'(z^m)| &+(1-\beta)|f(z^m)| + \sum_{n=1}^{\infty}|a_n|\phi_{n}(r)\\
		&\leq 	\beta f'_{0}(r^m)+ (1-\beta)f_{0}(r^m) +\sum_{n=1}^{\infty} M(n) \phi_{n}(r)\\
		& \leq d(0, \partial{\Omega}),
		\end{align*}
		which holds in $|z|=r\leq r_0$, where $r_0$ is the minimal positive root of the equation~\eqref{Gen-GBRequation1}. Existence of the root $r_0$ follows from the Intermediate value theorem for continuous function in $(0,1)$.
		To see the sharpness case, let us consider the function
		\begin{equation*}
		f_0(z)= z\exp\int_{0}^{z}\frac{\psi(t)-1}{t}dt.
		\end{equation*} 
		such that it's Taylor series coefficients $a_n(f_0)$ satisfies $|a_n(f_0)|=M(n)$. For this function we have 
		\begin{equation*}
		d(0, \partial f(\mathbb{D})) = -f_0(-1),
		\end{equation*}
		and the following equality holds for $|z|=r_0$:
		\begin{equation*}
		\beta f'_{0}(r^m)+ (1-\beta)f_{0}(r^m) +\sum_{n=1}^{\infty} M(n) \phi_{n}(r) = d(0, \partial f(\mathbb{D})), 
		\end{equation*}
		and therefore, if $f_0$ is extremal for each coefficient's bound, then the radius $r_0$ can not be improved. \qed	
	\end{proof}
	
	\begin{question}
		What if we do not have $M(n)$?
	\end{question}	 
	
	For such cases, the following result complements Theorem~\ref{Gen-Mainthm1}:
	\begin{theorem}\label{Gen-Mainthm2}
		Let $\{\phi_{n}(r)  \}_{n=1}^{\infty}$ be a non-negative sequence of continuous functions in $[0,1]$ such that
		\begin{equation*}
		\phi_1(r)+ \sum_{n=2}^{\infty} \left|\frac{f^{(n)}_{0}(0)}{n!}\right| \phi_{n}(r)
		\end{equation*}
		converges locally uniformly with respect to each $r\in[0,1)$.
		If 
		\begin{equation*}
		\beta |f'(z^m)| +(1-\beta)|f(z^m)|+\phi_{1}(r) +\sum_{n=2}^{\infty} \left|\frac{f^{(n)}_{0}(0)}{n!}\right| \phi_{n}(r) < -f_{0}(-1)
		\end{equation*}
		and
		$f(z)=z+\sum_{n=2}^{\infty}a_n z^n \in \mathcal{S}^*(\psi)$.
		Then
		\begin{equation}\label{Gen-Mainthm2-expr}
		\beta |f'(z^m)| +(1-\beta)|f(z^m)| + \sum_{n=1}^{\infty}|a_n|\phi_{n}(r) \leq d(0, \partial{\Omega})
		\end{equation}
		holds for $|z|=r\leq r_b=\{ 1/3, r_0 \}$, where $m\in \mathbb{N}$, $\Omega=f(\mathbb{D})$ and $r_0$ is the smallest positive root of the equation:
		\begin{equation*}
		\beta f'_{0}(r^m)+ (1-\beta)f_{0}(r^m) +\sum_{n=2}^{\infty} \left|\frac{f^{(n)}_{0}(0)}{n!}\right| \phi_{n}(r) = -f_{0}(-1)-\phi_{1}(r),
		\end{equation*}
		where
		\begin{equation*}
		f_0(z)= z\exp\int_{0}^{z}\frac{\psi(t)-1}{t}dt.
		\end{equation*}
		Moreover, the inequality \eqref{Gen-Mainthm2-expr} also holds for the class $S_{f}(\psi)$ in $|z|=r\leq r_b$. When $r_b=r_0$, then the radius is best possible.
	\end{theorem}
	\begin{proof}
		Since $f\in \mathcal{S}^*(\psi) $, it is known that
		\begin{equation*}
		\frac{f(z)}{z} \prec \frac{f_0(z)}{z}.
		\end{equation*}
		Applying Lemma~\ref{series-lem}, we see that
		\begin{equation*}
		\sum_{n=N}^{m} |a_n| |z|^n \leq \sum_{k=N}^{m} \left| \frac{f^k_{0}(0)}{k!}\right| |z|^k \quad \text{for}\quad |z|=r\leq \frac{1}{3}.
		\end{equation*}
		Now choosing $N=m$, we conclude that
		\begin{equation*}
		|a_n| \leq \left|\frac{f^n_{0}(0)}{n!} \right|
		\end{equation*}
		holds for each $n$ with in the disk $|z|=r\leq 1/3$.
		Hence, it suffices to see
		\begin{align*}
		&\beta |f'(z^m)| +(1-\beta)|f(z^m)| + \sum_{n=1}^{\infty}|a_n|\phi_{n}(r)\\
		\leq&\quad \beta |f'_{0}(z^m)| +(1-\beta)|f_{0}(z^m)|+\phi_{1}(r) +\sum_{n=2}^{\infty} \left|\frac{f^{(n)}_{0}(0)}{n!}\right| \phi_{n}(r)\\
		\leq &\quad -f_{0}(-1)\\
		\leq & \quad  d(0, \partial{\Omega}),
		\end{align*}
		holds in $|z|=r\leq \{r_0, 1/3\}$. If $r_0\leq 1/3$, then equatity case can be seen for the function $f=f_0$, whenever Taylor coefficients of $\psi$ are postive. \qed
	\end{proof}
	\begin{remark}
		\begin{enumerate}
			\item By taking $\phi_n(r)=r^n$ in Theorem~\ref{Gen-Mainthm2} give \cite[Theorem~5.1]{gangania-bohr}, and \cite[Theorem~3.1]{hamada-2021} for the choice $g=f$ with Taylor coefficients of $\psi$ being postive.
			\item Taking $\phi_n=r^n$ for $n\geq N$ and $0$ elsewhere in Theorem~\ref{Gen-Mainthm2} yields \cite[Theorem~5, Corollary~3]{kamal-mediter2021}.
		\end{enumerate}
	\end{remark}

	Let us discuss the generalized Bohr-Rogosinski phenomenon for the celebrated Janowski class of univalent starlike functions. For simplicity, write $\mathcal{S}^*((1+Dz)/(1+Ez))\equiv \mathcal{S}[D,E] $, where $-1\leq E<D\leq1$.
	
	\begin{corollary}\label{Jan-GBR}
		Let $\{\phi_{n}(r)  \}_{n=1}^{\infty}$ be a sequence of non-negative continuous functions in $(0,1)$ such that the series
		\begin{equation*}
		\phi_{1}(r)+\sum_{n=2}^{\infty} \prod_{k=0}^{n-2}\frac{|E-D+Ek|}{k+1}\phi_{n}(r)
		\end{equation*}
		converges locally uniformly with respect to each $r\in[0,1)$.
		If for $\beta\in[0,1]$
		\begin{equation}
		\beta f'_{0}(r^m)+ (1-\beta)f_{0}(r^m) +\sum_{n=1}^{\infty}|a_n(f_{0})|\phi_{n}(r) < -f_{0}(-1).
		\end{equation}
		and the function $f(z)=z+\sum_{n=2}^{\infty}a_n z^n \in \mathcal{S}[D,E]$.
		Then the following sharp inequality
		\begin{equation}\label{J-GBR-inequality}
		\beta |f'(z^m)| +(1-\beta)|f(z^m)| + \sum_{n=1}^{\infty}|a_n|\phi_{n}(r) \leq d(0, \partial{\Omega})
		\end{equation}
		holds for $|z|=r\leq r_0$, where $m\in \mathbb{N}$, $\Omega=f(\mathbb{D})$ and $r_0$ is the minimal positive root of the equations:\\
		If $E\neq0$
		\begin{align*}
		&\beta(1+Dr^m)(1+Er^m)^{\frac{D-2E}{E}} +(1-\beta)r^m(1+Er^m)^{\frac{D-E}{E}}\\
		&+\sum_{n=2}^{\infty} \prod_{k=0}^{n-2}\frac{|E-D+Ek|}{k+1}\phi_{n}(r) =(1-E)^{\frac{D-E}{E}} -\phi_{1}(r),
		\end{align*}
		and if $E=0$
		\begin{align*}
		e^{Dr^m}(\beta+(1-\beta(1-D))r^m)
		+\sum_{n=2}^{\infty} \prod_{k=0}^{n-2}\frac{D}{k+1}\phi_{n}(r) =e^{-D} -\phi_{1}(r),
		\end{align*}
		where
		\begin{equation}\label{Jan-exremal}
		f_0(z)= 
		\left\{
		\begin{array}
		{lr}
		z(1+E z)^{\frac{D-E}{E}}, & E\neq0; \\
		ze^{Dz},   & E=0.
		\end{array}
		\right.
		\end{equation}
		The radius $r_0$ can not be improved.
	\end{corollary}
	\begin{proof}
		Let $f(z)=z+\sum_{n=2}^{\infty}a_n z^n \in \mathcal{S}[D,E]$. Then for $n\geq2$, \cite[Theorem~3]{Aouf-1987} states that:
		\begin{equation*}
		|a_n| \leq \prod_{k=0}^{n-2}\frac{|E-D+Ek|}{k+1}=M(n),
		\end{equation*}	
		where the function $f_0$ given by \eqref{Jan-exremal} gives equality. The result now follows by Theorem~\ref{Gen-Mainthm1}. \qed

	\end{proof}
	
	\begin{remark}
		Taking $m\rightarrow \infty$ and $\beta=0$ in Corollary~\ref{Jan-GBR} yields: If 
		\begin{equation*}
		\sum_{n=1}^{\infty}|a_n(f_{0})|\phi_{n}(r) < -f_{0}(-1).
		\end{equation*}
		Then the sharp inequality
		\begin{equation*}
		\sum_{n=1}^{\infty}|a_n|\phi_{n}(r) \leq d(0, \partial{\Omega})
		\end{equation*}
		holds for $|z|=r\leq r_0$, where $m\in \mathbb{N}$, $\Omega=f(\mathbb{D})$ and $r_0$ is the minimal positive root of the equations:\\
		If $E\neq0$
		\begin{align*}
		\sum_{n=2}^{\infty} \prod_{k=0}^{n-2}\frac{|E-D+Ek|}{k+1}\phi_{n}(r) =(1-E)^{\frac{D-E}{E}} -\phi_{1}(r),
		\end{align*}
		and if $E=0$
		\begin{align*}
		\sum_{n=2}^{\infty} \prod_{k=0}^{n-2}\frac{D}{k+1}\phi_{n}(r) =e^{-D} -\phi_{1}(r),
		\end{align*}
		where $f_0$ is given in \eqref{Jan-exremal}.	 
	\end{remark}
	
	In Corollary~\ref{Jan-GBR}, putting $D=1-2\alpha$ and $E=-1$, where $0\leq\alpha<1$, we get the result for the class of univalent starlike functions of order $\alpha$, that is, $\mathcal{S}^{*}(\alpha)$:
	\begin{corollary}\label{starlikealpha-GBR}
		Let $\{\phi_{n}(r)  \}_{n=1}^{\infty}$ be a sequence of non-negative continuous functions in $(0,1)$ such that the series
		\begin{equation*}
		\phi_{1}(r)+\sum_{n=2}^{\infty} \prod_{k=0}^{n-2}\frac{k+2(1-\alpha)}{k+1}\phi_{n}(r)
		\end{equation*}
		converges locally uniformly with respect to each $r\in[0,1)$.
		If for $\beta\in[0,1]$
		\begin{align*}
		\frac{\beta(1+(1-2\alpha)r^m)}{(1-r^m)^{2(1-\alpha)+1}} &+\frac{(1-\beta)r^m}{(1-r^m)^{2(1-\alpha)}}
		+\sum_{n=2}^{\infty} \prod_{k=0}^{n-2}\frac{k+2(1-\alpha)}{k+1}\phi_{n}(r)\\ &<\frac{1}{4^{1-\alpha}} -\phi_{1}(r).
		\end{align*}
		and $f(z)=z+\sum_{n=2}^{\infty}a_n z^n \in \mathcal{S}^{*}(\alpha)$. Then the sharp inequality \eqref{J-GBR-inequality} holds for $|z|=r\leq r_0$, where $m\in \mathbb{N}$, $\Omega=f(\mathbb{D})$ and $r_0$ is the minimal positive root of the equations:
		\begin{align*}
		\frac{\beta(1+(1-2\alpha)r^m)}{(1-r^m)^{2(1-\alpha)+1}} &+\frac{(1-\beta)r^m}{(1-r^m)^{2(1-\alpha)}}
		+\sum_{n=2}^{\infty} \prod_{k=0}^{n-2}\frac{k+2(1-\alpha)}{k+1}\phi_{n}(r)\\ &=\frac{1}{4^{1-\alpha}} -\phi_{1}(r).
		\end{align*}
		The radius $r_0$ is sharp.
	\end{corollary}	
	
	Putting $\alpha=0$ in Corollary~\ref{starlikealpha-GBR}, we get the following:
	\begin{corollary}\label{starlike-GBR}
		Let the sequence $\{\phi_{n}(r)  \}_{n=1}^{\infty}$ satisfy the hypothesis of Corollary~\ref{starlikealpha-GBR} with $\alpha=0$. 
		If $f(z)=z+\sum_{n=2}^{\infty}a_n z^n \in \mathcal{S}^{*}$. Then the inequality \eqref{J-GBR-inequality} holds for $|z|=r\leq r_0$, where $m, N\in \mathbb{N}$, $\Omega=f(\mathbb{D})$ and $r_0$ is the smallest positive root of the equations:
		\begin{align*}
		\frac{\beta(1+r^m)}{(1-r^m)^{3}} +\frac{(1-\beta)r^m}{(1-r^m)^{2}}
		+\sum_{n=1}^{\infty}n\phi_{n}(r)=\frac{1}{4}.
		\end{align*}
		The radius $r_0$ is sharp.
	\end{corollary}	
	
	The following series of examples discuss the choices of sequence $\phi_{n}(r)$:
	\begin{example}
		Let us consider $\phi_{n}(r)=0$ for $1\leq n\leq N$, and $\phi_{n}(r)=r^n$ for $n\geq N$ in Corollary~\ref{Jan-GBR}. Then the following sharp inequality
		\begin{equation*}
		\beta |f'(z^m)| +(1-\beta)|f(z^m)| + \sum_{n=N}^{\infty}|a_n|r^n \leq d(0, \partial{\Omega})
		\end{equation*}
		holds for $|z|=r\leq R_{m,\beta,N}$, where $m, N\in \mathbb{N}$, $\Omega=f(\mathbb{D})$ and $R_{m,\beta,N}$ is the unique positive root of the equations:\\
		If $E\neq0$
		\begin{align*}
		&\beta(1+Dr^m)(1+Er^m)^{\frac{D-2E}{E}} +(1-\beta)r^m(1+Er^m)^{\frac{D-E}{E}}\\
		&+\sum_{n=N}^{\infty} \prod_{k=0}^{N-2}\frac{|E-D+Ek|}{k+1} r^n =(1-E)^{\frac{D-E}{E}},
		\end{align*}
		and if $E=0$
		\begin{align*}
		e^{Dr^m}(\beta+(1-\beta(1-D))r^m)
		+\sum_{n=2}^{\infty} \prod_{k=0}^{n-2}\frac{D}{k+1}\phi_{n}(r) =e^{-D}.
		\end{align*} 
	\end{example}
	
	\begin{example}
		Taking $\phi_{2n-1}(r)=0$ and $\phi_{2n}(r)=r^{2n}$ in Corollary~\ref{starlikealpha-GBR} yields
		\begin{equation*}
		\beta |f'(z^m)| +(1-\beta)|f(z^m)| + \sum_{n=1}^{\infty}|a_{2n}|r^{2n} \leq d(0, \partial{\Omega})
		\end{equation*}
		which holds for $|z|=r\leq R_{m,\beta,\alpha}$, where $m \in \mathbb{N}$, $\beta[0,1]$, $\Omega=f(\mathbb{D})$ and $R_{m,\beta,\alpha}$ is the unique positive root of the equations:
		\begin{align*}
		\frac{\beta(1+(1-2\alpha)r^m)}{(1-r^m)^{2(1-\alpha)+1}} &+\frac{(1-\beta)r^m}{(1-r^m)^{2(1-\alpha)}}
		+\sum_{n=1}^{\infty} \prod_{k=0}^{2(n-1)}\frac{k+2(1-\alpha)}{k+1}r^{2n}\\ &=\frac{1}{4^{1-\alpha}}.
		\end{align*}
		The radius is sharp.
	\end{example}
	
	\begin{example}
		Letting $\phi_{2n}(r)=0$ and $\phi_{2n-1}(r)=r^{2n-1}$ in Corollary~\ref{starlike-GBR} gives the following sharp inequality
		\begin{equation*}
		\beta |f'(z^m)| +(1-\beta)|f(z^m)| + \sum_{n=1}^{\infty}|a_{2n-1}|r^{2n-1} \leq d(0, \partial{\Omega})
		\end{equation*}
		in $|z|=r\leq R_{m,\beta}$ where $m \in \mathbb{N}$, $\beta[0,1]$, $\Omega=f(\mathbb{D})$ and $R_{m,\beta}$ is the unique positive root of the equations:
		\begin{align*}
		\frac{\beta(1+r^m)}{(1-r^m)^{2+1}} +\frac{(1-\beta)r^m}{(1-r^m)^{2}}
		+\frac{r(1+r^2)}{(1-r^2)^2}=\frac{1}{4}.
		\end{align*}
		The radius is sharp.
	\end{example}

	\section{Bohr-Rogosinski sum for starlike and convex functions with respect to conjugate and symmetric points}\label{sec-2Gen}

	\begin{definition}\label{BR-S-def0}
		Let us consider the subclass of close-to-convex univalent functions given by
		\begin{equation*}
		\mathcal{K}_{s}(\psi)= \left\{f\in \mathcal{A}: -\frac{z^2f'(z)}{h(z)h(-z)} \prec \psi(z)  \right\}
		\end{equation*}
		for some $h\in \mathcal{S}^*(1/2)$.
	\end{definition}
	\begin{theorem}\label{BR-S}
		Let $f\in \mathcal{K}_{s}(\psi)$ and $\Omega=f(\mathbb{D})$. If $g\in S_{f}(\psi)$. Then 
		\begin{equation}\label{BR-sf-Ineq}
		|g(z^m)| + \sum_{k=N}^{\infty}|b_k||z|^k \leq d(0, \partial{\Omega})
		\end{equation}
		holds for $|z|=r_b \leq \min \{ \frac{1}{3}, r_0 \}$, where $m, N\in \mathbb{N}$ and $r_0$ is the minimal positive root of the equation:
		\begin{equation}\label{BR-S-root1}
		\int_{0}^{r^m} \frac{\psi(t)}{1-t^2} + R^N(r)=\int_{0}^{1}\dfrac{\psi(-t)}{1+t^2} dt,
		\end{equation}
		where
		\begin{equation*}
		R^N(r)= \int_{0}^{r} \frac{M_{t}^{N}(\psi) t^{2N}}{t^2(1-t^2)}dt \quad \text{and} \quad f_0(z)=\int_{0}^{z} \frac{\psi(t)}{1-t^2}dt.
		\end{equation*}
	\end{theorem}
	\begin{proof}
		Let $g(z)=\sum_{k=1}^{\infty} b_k z^k \prec f(z)$. Then by Lemma~\ref{series-lem}, for $r\leq 1/3$
		\begin{equation}\label{BR-S-proof0}
		\sum_{k=N}^{\infty}|b_k| r^k  \leq \sum_{n=N}^{\infty} |a_n|r^n = \sum_{n=N}^{\infty} \frac{|\tilde{b}_n|}{n} r^n,
		\end{equation}
		where $\tilde{b}_n$ are the power series coefficient of $\tilde{G}(z)$ defined below. From Definition~\ref{BR-S-def0}, we have
		\begin{equation*}
		zf'(z)= G(z) \psi(\omega(z))=: \tilde{G}(z),
		\end{equation*}
		where
		\begin{equation*}
		G(z)= \frac{-h(z) h(-z)}{z} =: z+ \sum_{n=2}^{\infty} h_{2n-1} z^{2n-1},
		\end{equation*}
		which is an odd starlike function. Now a simple integration gives that
		\begin{equation*}
		f(z)=\int_{0}^{z} \frac{G(t) \psi(\omega(t))}{t} dt.
		\end{equation*}
		For $f\in \mathcal{A}$, let us consider the operator
		\begin{equation*}
		M_{r}^{N}(f) = \sum_{n=N}^{\infty} |a_n| |z^n| = \sum_{n=N}^{\infty} |a_n| r^n.
		\end{equation*}
		Then 
		\begin{equation}\label{BR-S-proof1}
		M_{r}^{N}( \tilde{G}) \leq M_{r}^{N}( G). M_{r}^{N}( \psi\circ \omega).
		\end{equation}
		Since $\psi\circ \omega \prec \psi$,  we get
		\begin{equation}\label{BR-S-proof2}
		M_{r}^{N}( \psi\circ \omega) \leq M_{r}^{N}( \psi), \quad\text{for}\quad r\leq 1/3.
		\end{equation}
		It is known that odd starlike functions statisfies $|h_{2n-1}|\leq 1$. Thus
		\begin{align}\label{BR-S-proof3}
		M_{r}^{N}( G) \leq \sum_{n=N}^{\infty} r^{2n-1}= \frac{1}{r} \left( \frac{r^{2N}}{1-r^2} \right).
		\end{align}
		Now combining the inequalities \eqref{BR-S-proof0}, \eqref{BR-S-proof1}, \eqref{BR-S-proof2} and \eqref{BR-S-proof3}, the following sequence of inequalities holds  for $r\leq 1/3$:
		\begin{align}
		\sum_{k=N}^{\infty}|b_k| r^k  &\leq  \sum_{n=N}^{\infty} \frac{|\tilde{b}_n|}{n} r^n 
		= \int_{0}^{r}\frac{M_{t}^{N}( \tilde{G})}{t}dt \nonumber\\
		& \leq \int_{0}^{r}\frac{M_{t}^{N}( G). M_{t}^{N}( \psi\circ \omega)}{t}dt  \nonumber\\
		& \leq \int_{0}^{r} \frac{M_{t}^{N}(\psi) t^{2N}}{t^2(1-t^2)}dt
		:= R^N(r). \label{BR-S-proof4}
		\end{align}
		Since, also see growth theorem in \cite[Theorem~2, page no.~4]{ChoKwonRavi-2011},
		\begin{align*}
		|g(z)|= |f(\omega(z))| \leq \max_{|z|=r} |f(|z|\leq r)|
		\leq \int_{0}^{r} \frac{\psi(t)}{1-t^2}dt
		= f_0(r) ,
		\end{align*}
		it follows that
		\begin{equation}\label{BR-S-proof5}
		|g(z^m)| \leq f(r^m) \leq \hat{f}_0(r^m), \quad \text{where}\quad f_0(z)=\int_{0}^{z} \frac{\psi(t)}{1-t^2}dt.
		\end{equation} 
		Finally, note that
		\begin{equation*}
		d(0, \Omega) \geq \int_{0}^{1}\frac{\psi(-t)}{1+t^2} dt.
		\end{equation*}
		Hence, from \eqref{BR-S-proof4} and \eqref{BR-S-proof5}, we get
		\begin{equation*}
		|g(z^m)| + \sum_{k=N}^{\infty}|b_k||z|^k \leq d(0, \partial{\Omega})
		\end{equation*}
		for $|z|=r\leq r_b= \{1/3, r_0\}$, where  $r_0 \in (0,1)$ is the minimal root of the equation~\eqref{BR-S-root1}. Existence of the root follows by Intermediate value theorem in the interval $[0,1]$.  \qed
	\end{proof}
	
	\begin{remark}
		Taking $m\rightarrow \infty$ and $N=1$, then Theorem~\ref{BR-S} reduces to \cite[Theorem~2.2]{AlluHalder-2021cc}.	
	\end{remark}	
	
	\begin{corollary}
		If $f\in \mathcal{K}_{s}(\psi)$ and $\Omega=f(\mathbb{D})$. Then 
		\begin{equation*}
		|f(z^m)| + \sum_{n=N}^{\infty}|a_n||z|^n \leq d(0, \partial{\Omega})
		\end{equation*}
		holds for $|z|=r_b \leq \min \{ \frac{1}{3}, r_0 \}$, where $m, N\in \mathbb{N}$ and $r_0$ is the minimal positive root of the equation:
		\begin{equation*}
		\int_{0}^{r^m} \frac{\psi(t)}{1-t^2} + R^N(r)=\int_{0}^{1}\dfrac{\psi(-t)}{1+t^2} dt,
		\end{equation*}
		where
		\begin{equation*}
		R^N(r)= \int_{0}^{r} \frac{M_{t}^{N}(\psi) t^{2N}}{t^2(1-t^2)}dt \quad \text{and} \quad f_0(z)=\int_{0}^{z} \frac{\psi(t)}{1-t^2}.
		\end{equation*}
	\end{corollary}
	
	\begin{definition}\label{S*cPsi}
		The class of starlike functions with respect to conjugate points given by
		\begin{equation*}
		\mathcal{S}^*_{c}(\psi)= \left\{f\in \mathcal{A}: \frac{zf'(z)}{f(z)+ \overline{f(\bar{z})} } \prec \psi(z)  \right\}.
		\end{equation*}
	\end{definition}
	\begin{lemma}\emph{(\cite{minda94})}\label{grth}
		Let $f\in \mathcal{S}^*(\psi)$ and $|z_0|=r<1$. Then $f(z)/z \prec f_0(z)/z$ and
		$$-f_0(-r)\leq|f(z_0)|\leq f_0(r).$$
		Equality holds for some $z_0\neq0$ if and only if $f$ is a rotation of $f_0$, where
		\begin{equation}\label{int-rep}
		f_0(z)=z\exp{\int_{0}^{z}\frac{\psi(t)-1}{t}dt}.
		\end{equation}
	\end{lemma}
	
	\begin{theorem}\label{BR-Convexcase}
		Let $h_{\psi}$ be given by \eqref{int-rep} and $f(z)=z+\sum_{n=2}^{\infty}a_n z^n \in \mathcal{S}^*_{c}(\psi)$. If $g\in S_{f}(\psi)$. Then 
		\begin{equation}\label{Sf-convex}
		|g(z^m)| + \sum_{k=N}^{\infty}|b_k||z|^k \leq d(0, \partial{\Omega})
		\end{equation}
		holds for $|z|=r_b \leq \min \{ \frac{1}{3}, r_0 \}$, where $m, N\in \mathbb{N}$, $\Omega=f(\mathbb{D})$ and $r_0$ is the unique positive root of the equation:
		\begin{equation}\label{BR-Convexcase-root0}
		h_{\psi}(r^m)+R^N(r)+h_{\psi}(-1)=0,
		\end{equation}
		where 
		\begin{equation*}
		R^N(r)=\int_{0}^{r}\frac{M_{t}^{N}(h_{\psi}). M_{t}^{N}(\psi) }{t}dt
		\end{equation*}
		The result is sharp when $r_b=r_0$ and $t_n>0$.
	\end{theorem}
	\begin{proof}
		Since the function $G(z)= (f(z)+ \overline{f(\bar{z})})/{2}$ belongs to $\mathcal{S}^*(\psi)$. Therefore, by Lemma~\cite{minda94} we have
		\begin{equation*}
		\frac{G(z)}{z} \prec \frac{h_{\psi}(z)}{z},
		\end{equation*}
		which using Lemma~\ref{series-lem} yields
		\begin{equation}\label{BR-Convexcase-proof0}
		M_{r}^{N}(G) \leq M_{r}^{N}(h_{\psi}) \quad \text{for} \quad r\leq \frac{1}{3}.
		\end{equation}
		From Definition~\ref{S*cPsi}, we get $zf'(z)=G(z)\psi(\omega(z))$ which after integration gives
		\begin{equation}\label{BR-Convexcase-proof1}
		f(z)=\int_{0}^{z}\frac{G(t) \psi(\omega(t))}{t}dt.
		\end{equation}
		Since $\psi\circ \omega \prec \psi$,
		\begin{equation}\label{BR-Convexcase-proof2}
		M_{r}^{N}(\psi\circ \omega) \leq M_{r}^{N}(\psi) \quad \text{for} \quad r\leq \frac{1}{3}.
		\end{equation}
		Thus, combining \eqref{BR-Convexcase-proof0}, \eqref{BR-Convexcase-proof1} and \eqref{BR-Convexcase-proof2}, we see that
		\begin{align*}
		\sum_{k=N}^{\infty}|b_k||z|^k &= M_{r}^{N}(g) \leq M_{r}^{N}(f) \\
		& = \int_{0}^{r} \frac{M_{t}^{N}(G) M_{t}^{N}(\psi\circ \omega)}{t} dt\\
		& \leq \int_{0}^{r} \frac{M_{t}^{N}(h_{\psi}) M_{t}^{N}(\psi)}{t} dt =: R^N(r),
		\end{align*}
		holds for $r\leq 1/3$. Also, using Maximum-principle of modulus and growth theorem~\cite{RaviConjugate-2004}, $g\prec f$ implies that
		\begin{align*}
		|g(|z|\leq r)|= |f(\omega(|z| \leq r))| \leq \max_{|z|=r}|f(|z|\leq r)| = h_{\psi}(r),
		\end{align*}
		which yields
		\begin{equation*}
		|g(z^m)| \leq h_{\psi}(r^m).
		\end{equation*}
		Finally, note that
		\begin{equation*}
		d(0, \Omega) \geq - h_{\psi}(-1).
		\end{equation*}
		Hence, 
		\begin{equation*}
		|g(z^m)| + \sum_{k=N}^{\infty}|b_k||z|^k \leq d(0, \partial{\Omega})
		\end{equation*}
		holds for $|z|\leq \{1/3, r_0\}$, where $r_0$ is the root of the equation~\eqref{BR-Convexcase-root0}. The existence of the root follows by Intermediate value theorem for continuous function in $[0,1]$. For the sharpness, note that for the function $h_{\psi}$
		\begin{equation*}
		d(0, \Omega) = - h_{\psi}(-1)
		\end{equation*}
		such that if $r_b=r_0$, then for the choice $g=f=h_{\psi}$:
		\begin{equation*}
		|h_{\psi}(z^m)| + \sum_{n=N}^{\infty}|t_n||z|^n = d(0, \partial{\Omega})
		\end{equation*}
		holds for $|z|=r_b$ with $t_n>0$, where $h_{\psi}(z)=z+\sum_{n=2}^{\infty}t_n z^n$ as given in \eqref{int-rep}. \qed
	\end{proof}	
	
	\begin{remark}
		Let $\psi(z)=(1+z)/(1-z)$, then Theorem~\ref{BR-Convexcase} reduces to \cite[Theorem~6]{Kayumov-2021}.
	\end{remark}	
	
	The following result is explicitly for the class $\mathcal{S}^*_{c}(\psi)$.
	\begin{corollary}\label{convexclass-BR}
		Let $f(z)=z+\sum_{n=2}^{\infty}a_n z^n \in \mathcal{S}^*_{c}(\psi)$. Then 
		\begin{equation}\label{convexclass-BR-expr0}
		|f(z^m)| + \sum_{n=N}^{\infty}|a_n||z|^n \leq d(0, \partial{\Omega})
		\end{equation}
		holds for $|z|=r_b \leq \min \{ \frac{1}{3}, r_0 \}$, where $m, N\in \mathbb{N}$, $\Omega=f(\mathbb{D})$ and $r_0$ is the unique positive root of the equation:
		\begin{equation*}
		h_{\psi}(r^m)+R^N(r)+h_{\psi}(-1)=0,
		\end{equation*}
		where 
		\begin{equation*}
		R^N(r)=\int_{0}^{r}\frac{M_{t}^{N}(h_{\psi}). M_{t}^{N}(\psi) }{t}dt
		\end{equation*}
		and  $h_{\psi}$ be given by \eqref{int-rep}. The result is sharp when $r_b=r_0$ and $t_n>0$.
	\end{corollary}	
	\begin{remark}
		Taking $m\rightarrow \infty$ and $N=1$ in Theorem~\ref{BR-Convexcase} and Corollary~\ref{convexclass-BR} establish the Bohr phenonmenon for the classes $S_{f}(\psi)$ and $\mathcal{S}^*_{c}(\psi)$, respectively given in \cite[Lemma~2.12]{AlluHalder-2021cc} and \cite[Theorem~2.9]{AlluHalder-2021cc}.
	\end{remark}	
	
	To proceed further, we need to recall the following fundamental result:
	\begin{lemma}\label{C-grth}\cite{minda94}
		Let $f\in \mathcal{C}(\psi)$. Then $zf''(z)/f'(z) \prec zl_{0}''(z)/l_{0}'(z)$ and $f'(z)\prec l_{0}'(z)$. Also, for $|z|=r$ we have 
		$$-l_{0}(-r)\leq |f(z)| \leq l_{0}(r),$$
		where 
		\begin{equation}\label{int-rep-C}
		1+zl_{0}''(z)/l_{0}'(z)=\psi(z).
		\end{equation}
	\end{lemma}	
	
	\begin{definition}\label{CcPsi}
		The class of convex functions with respect to conjugate points given by
		\begin{equation*}
		\mathcal{C}_{c}(\psi)= \left\{f\in \mathcal{A}: \frac{(zf'(z))'}{(f(z)+ \overline{f(\bar{z})})' } \prec \psi(z)  \right\}.
		\end{equation*}
	\end{definition}
	
	\begin{theorem}\label{BR-CcPsi}
		Let $f(z)=z+\sum_{n=2}^{\infty}a_n z^n \in \mathcal{C}_{c}(\psi)$. If $g\in S_{f}(\psi)$. Then 
		\begin{equation*}
		|g(z^m)| + \sum_{k=N}^{\infty}|b_k||z|^k \leq d(0, \partial{\Omega})
		\end{equation*}
		holds for $|z|=r_b \leq \min \{ \frac{1}{3}, r_0 \}$, where $m, N\in \mathbb{N}$, $\Omega=f(\mathbb{D})$ and $r_0$ is the minimal positive root of the equation:
		\begin{equation}\label{BR-CcPsi-root0}
		k_{\psi}(r^m)+R^N(r)+k_{\psi}(-1)=0,
		\end{equation}
		where 
		\begin{equation*}
		R^N(r)=\int_{0}^{r}\frac{1}{s}\int_{0}^{s}{M_{t}^{N}(k'_{\psi}). M_{t}^{N}(\psi) }dtds,
		\end{equation*}
		and $k_{\psi}(z)=z+\sum_{n=2}^{\infty}l_n z^n$ is given by \eqref{int-rep-C}.
		The result is sharp when $r_b=r_0$ and $l_n>0$.
	\end{theorem}
	\begin{proof}
		Consider the function
		\begin{equation*}
		G(z)=\frac{f(z)+\overline{f(\bar{z})}}{2}. 
		\end{equation*}
		Then $G\in \mathcal{C}(\psi) $. Now from Definition~\ref{CcPsi}, we see that 
		\begin{equation}\label{BR-CcPsi-proof0}
		(zf'(z))'= G'(z)\psi(\omega(z)).
		\end{equation}
		This gives
		\begin{equation}\label{BR-CcPsi-proof1}
		f(z)= \int_{0}^{z}\frac{1}{y} \int_{0}^{y} G'(t) \psi(\omega(t)) dt dy.
		\end{equation}
		As $G' \prec k'_{\psi}$, see Lemma~\ref{C-grth}, it follows using Lemma~\ref{series-lem} that
		\begin{equation}\label{BR-CcPsi-proof2}
		M_{r}^{N}(G') \leq M_{r}^{N}(k'_{\psi}) \quad \text{for} \quad r\leq \frac{1}{3}.
		\end{equation}
		Hence, using \eqref{BR-CcPsi-proof0}, \eqref{BR-CcPsi-proof1} and \eqref{BR-CcPsi-proof2}
		\begin{align*}
		|g(z^m)| + M_{r}^{N}(f) &\leq k_{\psi}(r^m)+ \int_{0}^{r}\frac{1}{y} \int_{0}^{y} M_{t}^{N}(G') M_{t}^{N}(\psi\circ\omega) dt dy\\
		&\leq k_{\psi}(r^m)+ \int_{0}^{r}\frac{1}{y} \int_{0}^{y} M_{t}^{N}(k'_{\psi}) M_{t}^{N}(\psi) dt dy \\
		& \leq -k_{\psi}(-1)\\
		& \leq d(0, \partial\Omega),
		\end{align*}
		holds for $|z|=r\leq r_b=\{1/3, r_0\} $, where $r_0$ is minimal root of the equation~\eqref{BR-CcPsi-root0}. The existence of $0<r_0<1$ can be seen by Intermediate value theorem for continuous function in $[0,1]$. The case of equality 
		\begin{equation*}
		|g(z^m)| + \sum_{k=N}^{\infty}|b_k||z|^k = d(0, \partial{\Omega})
		\end{equation*}
		follows with the choice $g=f=k_{\psi}$, whenever $r_b=r_0$ and $l_n>0$. \qed
	\end{proof}
	
	\begin{corollary}\label{CcPsi-corollary}
		If $f(z)=z+\sum_{n=2}^{\infty}a_n z^n \in \mathcal{C}_{c}(\psi)$. Then 
		\begin{equation*}
		|f(z^m)| + \sum_{k=N}^{\infty}|a_n||z|^n \leq d(0, \partial{\Omega})
		\end{equation*}
		holds for $|z|=r_b \leq \min \{ \frac{1}{3}, r_0 \}$, where $m, N\in \mathbb{N}$, $\Omega=f(\mathbb{D})$ and $r_0$ is the minimal positive root of the equation:
		\begin{equation*}
		k_{\psi}(r^m)+R^N(r)+k_{\psi}(-1)=0,
		\end{equation*}
		where 
		\begin{equation*}
		R^N(r)=\int_{0}^{r}\frac{1}{s}\int_{0}^{s}{M_{t}^{N}(k'_{\psi}). M_{t}^{N}(\psi) }dtds.
		\end{equation*}
		The result is sharp when $r_b=r_0$ and $l_n>0$.
	\end{corollary}
	\begin{remark}
		Taking $m\rightarrow\infty$ and $N=1$ in Corollary~\ref{CcPsi-corollary} gives \cite[Theorem~2.23]{AlluHalder-2021cc}.
	\end{remark}
	
	\begin{definition}
		The class of convex function with respect to symmetric points is given by
		\begin{equation*}
		\mathcal{C}_{s}(\psi) = \left\{ f\in \mathcal{A} : \frac{2(zf'(z))'}{f'(z)+f'(-z)} \prec \psi(z)    \right\}.
		\end{equation*}
	\end{definition}
	
	Now, we omit the details of the proof as it works on the similar lines discussed in the above theorems.
	\begin{theorem}\label{CsPsi}
		Let $k_{\psi}(z)=z+\sum_{n=2}^{\infty}l_n z^n$ be given by \eqref{int-rep-C} and $f(z)=z+\sum_{n=2}^{\infty}a_n z^n \in \mathcal{C}_{s}(\psi)$. If $g\in S_{f}(\psi)$. Then 
		\begin{equation*}
		|g(z^m)| + \sum_{k=N}^{\infty}|b_k||z|^k \leq d(0, \partial{\Omega})
		\end{equation*}
		holds for $|z|=r_b \leq \min \{ \frac{1}{3}, r_0 \}$, where $m, N\in \mathbb{N}$, $\Omega=f(\mathbb{D})$ and $r_0$ is the unique positive root of the equation:
		\begin{equation*}
		\int_{0}^{r}\frac{1}{s}\int_{0}^{s}\psi(t)(k'_{\psi}(t^2))^{1/2}dtds + R^N(r)= \int_{0}^{1}\frac{1}{s} \int_{0}^{s} \psi(-t)(k'_{\psi}(-t^2))^{1/2}dtds,
		\end{equation*}
		where $K'(z)=(k'_{\psi}(z^2))^{1/2}$ and
		\begin{equation*}
		R^N(r)=\int_{0}^{r}\frac{1}{s}\int_{0}^{s}{M_{t}^{N}(K'). M_{t}^{N}(\psi) }dtds.
		\end{equation*}
		The result is sharp when $r_b=r_0$ and $l_n>0$.
	\end{theorem}
	
	\begin{corollary}\label{CsPsi-corollary}
		If $f(z)=z+\sum_{n=2}^{\infty}a_n z^n \in \mathcal{C}_{s}(\psi)$. Then 
		\begin{equation*}
		|f(z^m)| + \sum_{k=N}^{\infty}|a_n||z|^n \leq d(0, \partial{\Omega})
		\end{equation*}
		holds for $|z|=r_b \leq \min \{ \frac{1}{3}, r_0 \}$, where $m, N\in \mathbb{N}$, $\Omega=f(\mathbb{D})$ and $r_0$ is as given in Theorem~\ref{CsPsi}.
		The result is sharp when $r_b=r_0$ and $l_n>0$.
	\end{corollary}
	\begin{remark}
		Taking $m\rightarrow\infty$ and $N=1$ in Corollary~\ref{CsPsi-corollary} gives \cite[Theorem~2.25]{AlluHalder-2021cc}.
	\end{remark}	
	
	\section*{Conflict of interest}
	The authors declare that they have no conflict of interest.

\end{document}